\documentclass[english]{article}
\usepackage[charter]{mathdesign}
\usepackage[T1]{fontenc}
\usepackage[latin9]{inputenc}
\usepackage[letterpaper]{geometry}
\usepackage{fancyhdr}
\usepackage{babel}
\usepackage{amsmath}
\usepackage{amsthm}
\usepackage[all]{xy}
\usepackage[unicode=true,
 bookmarks=true,bookmarksnumbered=false,bookmarksopen=false,
 breaklinks=false,pdfborder={0 0 1},backref=false,colorlinks=false]
 {hyperref}

\makeatletter
\numberwithin{equation}{section}
\numberwithin{figure}{section}
\theoremstyle{plain}
\newtheorem{thm}{\protect\theoremname}
  \theoremstyle{plain}
  \newtheorem{conjecture}[thm]{\protect\conjecturename}
  \theoremstyle{definition}
  \newtheorem{defn}[thm]{\protect\definitionname}
  \theoremstyle{plain}
  \newtheorem{cor}[thm]{\protect\corollaryname}
  \theoremstyle{plain}
  \newtheorem{prop}[thm]{\protect\propositionname}

\newcommand\xyR[1]{\xydef@\xymatrixrowsep@{#1}}
\newcommand\xyC[1]{\xydef@\xymatrixcolsep@{#1}}

\newcommand\thistheoremname{}
\newtheorem{genericthm}[thm]{\thistheoremname}
\newenvironment{namedthm}[1]
{\renewcommand\thistheoremname{#1}\begin{genericthm}}
{\end{genericthm}}

\addto\captionsenglish{\renewcommand\bibname{References}}

\date{}

\makeatother

  \providecommand{\conjecturename}{Conjecture}
  \providecommand{\corollaryname}{Corollary}
  \providecommand{\definitionname}{Definition}
  \providecommand{\propositionname}{Proposition}
\providecommand{\theoremname}{Theorem}

\begin{document}
\lhead{Division Polynomials and Intersection of Projective Torsion Points}\rhead{Bogomolov and Fu}

\title{\textsc{Division Polynomials and Intersection}\\
\textsc{of Projective Torsion Points}}

\author{\textsc{Fedor Bogomolov} and \textsc{Hang Fu}}
\maketitle
\begin{quote}
\textbf{\small{}Abstract.}{\small{} Given two elliptic curves, each
of which is associated with a projection map that identifies opposite
elements with respect to the natural group structure, we investigate
how their corresponding projective images of torsion points intersect.}{\small \par}

\textbf{\small{}Keywords.}{\small{} Elliptic curve $\cdot$ Division
polynomial $\cdot$ Unlikely intersection $\cdot$ Jordan's totient
function}{\small \par}

\textbf{\small{}Mathematics Subject Classification.}{\small{} 14H52}{\small \par}
\end{quote}

\section{Introduction}

Throughout the article, let $K$ be a field of characteristic $0$,
$(E,O)$ an elliptic curve defined over $K$, $E[n]$ the collection
of $n$-th torsion points, $E^{*}[n]$ the collection of torsion points
of exact order $n$, $E[\infty]$ the collection of all torsion points,
and $\pi\in K(E)$ an even morphism of degree $2$.

\medskip{}

Bogomolov and Tschinkel \cite{1} observed that
\begin{thm}
\label{theorem 1}If $\pi_{1}(E_{1}[2])$ and $\pi_{2}(E_{2}[2])$
are different, then\textup{ }the intersection of $\pi_{1}(E_{1}[\infty])$
and $\pi_{2}(E_{2}[\infty])$ is finite.\end{thm}
\begin{proof}
Consider the product map $\pi_{1}\times\pi_{2}:E_{1}\times E_{2}\rightarrow\mathbb{P}^{1}\times\mathbb{P}^{1}$,
let $\Delta$ be the diagonal curve. Suppose that $\#\pi_{1}(E_{1}[2])\cap\pi_{2}(E_{2}[2])=0\text{ }(\text{resp. }1,2,3)$,
then the preimage $(\pi_{1}\times\pi_{2})^{-1}(\Delta)$ is a curve
of genus $5\text{ }(\text{resp. }4,3,2)$, and hence contains only
finitely many torsion points of $E_{1}\times E_{2}$ by Raynaud's
theorem \cite{10}.
\end{proof}
However, we expect not only the finiteness, but also the existence
of a universal bound of the cardinality of their intersection.
\begin{conjecture}
\label{conjeture 2}
\[
\underset{\{(K,E_{1},O_{1},\pi_{1},E_{2},O_{2},\pi_{2}):\pi_{1}(E_{1}[2])\neq\pi_{2}(E_{2}[2])\}}{\textup{sup}}\#\pi_{1}(E_{1}[\infty])\cap\pi_{2}(E_{2}[\infty])<\infty
\]

\end{conjecture}
Here the supremum is taken over all $K$, but clearly we can restrict
$K=\bar{\mathbb{\mathbb{Q}}}$ or $K=\mathbb{C}$. In section \ref{section 3},
theorem \ref{theorem 12} and \ref{theorem 13} will indicate that
under some mild conditions, the cardinality is small, while theorem
\ref{theorem 18} will give a construction to show that it can be
at least $14$. (See also remark \ref{remark 19}.) The main tool
to achieve these results is the explicit calculation of division polynomials,
which will be introduced and developed in section \ref{section 2}.
Jordan's totient function, as an ingredient of division polynomials,
will be briefly discussed in the appendix. Calculations were assisted
by Mathematica 10.0 \cite{14}.

\section{\label{section 2}Division Polynomials}

Now let $E:y^{2}+a_{1}xy+a_{3}y=x^{3}+a_{2}x^{2}+a_{4}x+a_{6}$ be
in generalized Weierstrass form with identity element $O^{W}=(0:1:0)$,
then it has a canonical projection map $\pi^{W}:E\rightarrow\mathbb{P}^{1},(x,y)\mapsto x$.
We have the standard quantities
\begin{eqnarray*}
b_{2} & = & a_{1}^{2}+4a_{2},\\
b_{4} & = & 2a_{4}+a_{1}a_{3},\\
b_{6} & = & a_{3}^{2}+4a_{6},\\
b_{8} & = & a_{1}^{2}a_{6}+4a_{2}a_{6}-a_{1}a_{3}a_{4}+a_{2}a_{3}^{2}-a_{4}^{2},\\
c_{4} & = & b_{2}^{2}-24b_{4},\\
c_{6} & = & -b_{2}^{3}+36b_{2}b_{4}-216b_{6},\\
\Delta & = & -b_{2}^{2}b_{8}-8b_{4}^{3}-27b_{6}^{2}+9b_{2}b_{4}b_{6},\\
j & = & c_{4}^{3}/\Delta,
\end{eqnarray*}
with relations
\begin{eqnarray*}
4b_{8} & = & b_{2}b_{6}-b_{4}^{2},\\
1728\Delta & = & c_{4}^{3}-c_{6}^{2}.
\end{eqnarray*}
Traditionally, the division polynomials $\psi_{n}$ \cite{12} are
defined by the initial values
\begin{eqnarray*}
\psi_{1} & = & 1,\\
\psi_{2} & = & 2y+a_{1}x+a_{3},\\
\psi_{3} & = & 3x^{4}+b_{2}x^{3}+3b_{4}x^{2}+3b_{6}x+b_{8},\\
\psi_{4} & = & \psi_{2}\cdot(2x^{6}+b_{2}x^{5}+5b_{4}x^{4}+10b_{6}x^{3}+10b_{8}x^{2}+(b_{2}b_{8}-b_{4}b_{6})x+(b_{4}b_{8}-b_{6}^{2})),
\end{eqnarray*}
and the inductive formulas
\begin{eqnarray*}
\psi_{2n+1} & = & \psi_{n}^{3}\psi_{n+2}-\psi_{n-1}\psi_{n+1}^{3}\text{ for }n\geq2,\\
\psi_{2}\psi_{2n} & = & \psi_{n-1}^{2}\psi_{n}\psi_{n+2}-\psi_{n-2}\psi_{n}\psi_{n+1}^{2}\text{ for }n\geq3.
\end{eqnarray*}
Notice that
\[
\psi_{2}^{2}=4x^{3}+b_{2}x^{2}+2b_{4}x+b_{6}.
\]

Since $\text{char}(K)=0$, we can eliminate $b_{8}$ and the leading
coefficients.
\begin{defn}
Let $n>1$,\\
(A) the normalized $n$-th division polynomial
\[
f_{n}(x)=\prod_{\{P:P\in E[n]\setminus\{O^{W}\}\}}(x-\pi^{W}(P));
\]
(B) the normalized $n$-th primitive division polynomial
\[
F_{n}(x)=\prod_{\{\pi^{W}(P):P\in E^{*}[n]\}}(x-\pi^{W}(P)).
\]
\end{defn}
\begin{thm}
We have the following explicit formulas:\\
\textup{(A)}
\[
f_{n}(x)=\sum_{\{(r,s,t):r+2s+3t\leq d(n)\}}c_{r,s,t}(n)b_{2}^{r}b_{4}^{s}b_{6}^{t}x^{d(n)-(r+2s+3t)},
\]
the degree and the first three coefficients are 
\begin{eqnarray*}
d(n) & = & n^{2}-1,\\
c_{1,0,0}(n) & = & \frac{n^{2}-1}{12}=\frac{1}{12}n^{2}-\frac{1}{12},\\
c_{0,1,0}(n) & = & \frac{(n^{2}-1)(n^{2}+6)}{60}=\frac{1}{60}n^{4}+\frac{1}{12}n^{2}-\frac{1}{10},\\
c_{0,0,1}(n) & = & \frac{(n^{2}-1)(n^{4}+n^{2}+15)}{420}=\frac{1}{420}n^{6}+\frac{1}{30}n^{2}-\frac{1}{28};
\end{eqnarray*}
\textup{(B)}
\[
F_{n}(x)=\sum_{\{(r,s,t):r+2s+3t\leq D(n)\}}C_{r,s,t}(n)b_{2}^{r}b_{4}^{s}b_{6}^{t}x^{D(n)-(r+2s+3t)},
\]
the degree and the first three coefficients are
\begin{eqnarray*}
D(n) & = & \frac{1}{2}J_{2}(n)I(n),\\
C_{1,0,0}(n) & = & \frac{1}{24}J_{2}(n)I(n),\\
C_{0,1,0}(n) & = & \left(\frac{1}{120}J_{4}(n)+\frac{1}{24}J_{2}(n)\right)I(n),\\
C_{0,0,1}(n) & = & \left(\frac{1}{840}J_{6}(n)+\frac{1}{60}J_{2}(n)\right)I(n),
\end{eqnarray*}
where
\[
I(n)=\begin{cases}
2, & \textup{if }n=2,\\
1, & \textup{if }n>2,
\end{cases}
\]
and 
\[
J_{k}(n)=n^{k}\prod_{p|n}\left(1-\frac{1}{p^{k}}\right)
\]
is Jordan's totient function.\end{thm}
\begin{proof}
(A) Notice that $f_{n}(x)=\psi_{n}^{2}(x)/n^{2}$, so the inductive
formulas for $\psi_{n}$ can be transformed to
\begin{eqnarray*}
f_{2n+1} & = & \left(\frac{n^{3}(n+2)}{2n+1}\sqrt{f_{n}^{3}f_{n+2}}-\frac{(n-1)(n+1)^{3}}{2n+1}\sqrt{f_{n-1}f_{n+1}^{3}}\right)^{2}\text{ for }n\geq2,\\
f_{2n} & = & \left(\frac{(n-1)^{2}(n+2)}{4}\sqrt{\frac{f_{n-1}^{2}f_{n}f_{n+2}}{f_{2}}}-\frac{(n-2)(n+1)^{2}}{4}\sqrt{\frac{f_{n-2}f_{n}f_{n+1}^{2}}{f_{2}}}\right)^{2}\text{ for }n\geq3.
\end{eqnarray*}
In order to use induction to prove the formulas for $t(n)=d(n),$
$c_{1,0,0}(n)$, $c_{0,1,0}(n)$, and $c_{0,0,1}(n)$, we need to
check the initial values $t(n)$ for $1\leq n\leq4$, and verify that
they all satisfy
\begin{eqnarray*}
t(2n+1) & = & \frac{n^{3}(n+2)}{2n+1}(3t(n)+t(n+2))-\frac{(n-1)(n+1)^{3}}{2n+1}(t(n-1)+3t(n+1))\text{ for }n\geq2,\\
t(2n) & = & \frac{(n-1)^{2}(n+2)}{4}(2t(n-1)+t(n)+t(n+2)-t(2))\\
 &  & -\frac{(n-2)(n+1)^{2}}{4}(t(n-2)+t(n)+2t(n+1)-t(2))\text{ for }n\geq3.
\end{eqnarray*}
All of these can be easily done.\\
(B) By definition,
\[
f_{n}(x)=\prod_{d|n,d\neq1}F_{d}(x)^{2/I(d)},
\]
so for $T(n)=D(n)$, $C_{1,0,0}(n)$, $C_{0,1,0}(n)$, and $C_{0,0,1}(n)$,
we have 
\[
t(n)=\sum_{d|n,d\neq1}\frac{2}{I(d)}T(d).
\]
Then the rest is a standard application of M\"obius inversion formula.\end{proof}
\begin{thm}
\label{theorem 5}Let $E_{1}$ and $E_{2}$ be two elliptic curves
in generalized Weierstrass form, then the following are equivalent:\\
\textup{(A)} $b_{i}(E_{1})=b_{i}(E_{2}),i=2,4,6$;\\
\textup{(B)} for any $n$, $\pi^{W}(E_{1}^{*}[n])=\pi^{W}(E_{2}^{*}[n])$;\\
\textup{(C)} $\pi^{W}(E_{1}[\infty])=\pi^{W}(E_{2}[\infty])$;\\
\textup{(D)} $\pi^{W}(E_{1}[\infty])\cap\pi^{W}(E_{2}[\infty])$ is
infinite;\\
\textup{(E)} for some $n>1$, $\pi^{W}(E_{1}^{*}[n])=\pi^{W}(E_{2}^{*}[n])$;\\
\textup{(F)} for some $n_{1},\ldots,n_{k}>1$, $\pi^{W}(\cup_{i=1}^{k}E_{1}^{*}[n_{i}])=\pi^{W}(\cup_{i=1}^{k}E_{2}^{*}[n_{i}])$.\end{thm}
\begin{proof}
(A)$\Rightarrow$(B)$\Rightarrow$(C)$\Rightarrow$(D)$\Rightarrow$(E)$\Rightarrow$(F)
is clear, where (D)$\Rightarrow$(E) is given by theorem \ref{theorem 1}.
Assume (F), then $E_{1}$ and $E_{2}$ share the same $\prod_{i=1}^{k}F_{n_{i}}(x)$.
Since $C_{1,0,0}(n)$, $C_{0,1,0}(n)$, and $C_{0,0,1}(n)$ are all
strictly positive, the coefficients of $b_{2}$, $b_{4}$, and $b_{6}$
in the product $\prod_{i=1}^{k}F_{n_{i}}(x)$ will always be nonzero,
then $b_{2}$, $b_{4}$, and $b_{6}$ can be solved, which proves
(A).
\end{proof}
Now let us go back to the general $(E,O,\pi)$, and assume that $\pi(O)=\infty$.
By Riemann-Roch theorem, there exists an isomorphism $\phi:E\rightarrow E'$
such that $E'$ is in generalized Weierstrass form. Then $\phi$ induces
$\bar{\phi}\in\text{Aut}(\mathbb{\mathbb{P}}^{1})$ fixing $\infty$,
which must be a linear function \cite{5}. Its inverse $\bar{\phi}^{-1}$
can be lifted to $\hat{\phi}:E'\rightarrow E''$ such that $E''$
remains in generalized Weierstrass form. Thus the general case can
be reduced to the Weierstrass case. Note that we can make everything
above defined over $K$, except possibly $\hat{\phi}$ has to be defined
over a quadratic extension of $K$. 
\[
\xymatrix{(E,O)\ar[r]^{\phi}\ar[d]^{\pi} & (E',O^{W})\ar[r]^{\hat{\phi}}\ar[d]^{\pi^{W}} & (E'',O^{W})\ar[d]^{\pi^{W}}\\
(\mathbb{\mathbb{P}}^{1},\infty)\ar[r]^{\bar{\phi}}\ar@/_{0.25in}/[rr]_{id} & (\mathbb{\mathbb{P}}^{1},\infty)\ar[r]^{\bar{\phi}^{-1}} & (\mathbb{\mathbb{P}}^{1},\infty)
}
\]

\begin{cor}
\label{corollary 6}Given $(E_{1},O_{1},\pi_{1})$ and $(E_{2},O_{2},\pi_{2})$
such that $\pi_{1}(O_{1})=\pi_{2}(O_{2})$, then the following are
equivalent:\\
\textup{(A)} for any $n$, $\pi_{1}(E_{1}^{*}[n])=\pi_{2}(E_{2}^{*}[n])$;\\
\textup{(B)} for some $n>1$, $\pi_{1}(E_{1}^{*}[n])=\pi_{2}(E_{2}^{*}[n])$;\\
\textup{(C)} for some $n_{1},\ldots,n_{k}>1$,$\pi_{1}(\cup_{i=1}^{k}E_{1}^{*}[n_{i}])=\pi_{2}(\cup_{i=1}^{k}E_{2}^{*}[n_{i}])$;\\
\textup{(D)} $\pi_{1}(E_{1}[\infty])=\pi_{2}(E_{2}[\infty])$;\\
\textup{(E)} $\pi_{1}(E_{1}[\infty])\cap\pi_{2}(E_{2}[\infty])$ is
infinite.\end{cor}
\begin{proof}
First move $\pi_{1}(O_{1})$ and $\pi_{2}(O_{2})$ to $\infty$ by
a common fractional linear transformation, and then use the above
argument.
\end{proof}
The following corollary gives the converse of theorem \ref{theorem 1}.
\begin{cor}
Given $(E_{1},O_{1},\pi_{1})$ and $(E_{2},O_{2},\pi_{2})$, then
the following are equivalent:\\
\textup{(A)} $\pi_{1}(E_{1}[2])=\pi_{2}(E_{2}[2])$;\\
\textup{(B)} $\pi_{1}(E_{1}[\infty])=\pi_{2}(E_{2}[\infty])$;\\
\textup{(C)} $\pi_{1}(E_{1}[\infty])\cap\pi_{2}(E_{2}[\infty])$ is
infinite.\end{cor}
\begin{proof}
Assume (A), if $\pi_{1}(O_{1})=\pi_{2}(P)$ for some $P\in E_{2}[2]$,
then the translation-by-$P$ map $[+P]$ acts on $E_{2}[2]$ and $E_{2}[\infty]$
bijectively, the projection map $\pi_{2}\circ[+P]$ satisfies $\pi_{1}(O_{1})=\pi_{2}\circ[+P](O_{2})$,
which implies (B) by corollary \ref{corollary 6}. (B)$\Rightarrow$(C)
is obvious, and (C)$\Rightarrow$(A) is a restatement of theorem \ref{theorem 1}.
\end{proof}
Given theorem \ref{theorem 5}, it is natural to ask

\begin{namedthm}{Question}\textup{\label{question 8}Let $E_{1}$
and $E_{2}$ be two elliptic curves in generalized Weierstrass form,
if $\pi^{W}(E_{1}^{*}[n_{1}])=\pi^{W}(E_{2}^{*}[n_{2}])$, can we
always conclude that $n_{1}=n_{2}$?}\end{namedthm}

If conjecture \ref{conjeture 2} is true, then the answer must be
yes, at least when $n_{1}$ and $n_{2}$ are large enough. A naive
attempt is to write their division polynomials more explicitly, and
then compare those coefficients. Clearly, a necessary condition is
$D(n_{1})=D(n_{2})$. However, since we have, for example,
\begin{eqnarray*}
 &  & D(5)=D(6)=12,\\
 &  & D(35)=D(40)=D(42)=576,\\
 &  & D(55)=D(57)=D(62)=D(66)=1440,
\end{eqnarray*}
$D(n)$ itself does not give a strong restriction on $n$. Notice
that $C_{1,0,0}(n)=D(n)/12$, so $b_{2}(E_{1})=b_{2}(E_{2})$, we
can assume the common value is $0$, and consequently $b_{2}$ disappears
in the formulas of $f_{n}(x)$ and $F_{n}(x)$. Based on the same
approach as before, together with a tedious calculation, we can obtain
all the coefficients $C_{0,s,t}(n)$ for $2s+3t\leq6$.
\begin{thm}
We have the following explicit formulas for the coefficients of $F_{n}(x)$:
\begin{eqnarray*}
C_{0,2,0}(n) & = & -\left(\frac{1}{16800}J_{8}(n)+\frac{1}{600}J_{4}(n)+\frac{5}{672}J_{2}(n)\right)I(n)\\
 &  & +\frac{1}{2}\left(\frac{1}{120}J_{4}(n)+\frac{1}{24}J_{2}(n)\right)^{2}I^{2}(n),\\
C_{0,1,1}(n) & = & -\left(\frac{1}{92400}J_{10}(n)+\frac{1}{2800}J_{6}(n)+\frac{1}{1680}J_{4}(n)+\frac{1}{150}J_{2}(n)\right)I(n)\\
 &  & +\left(\frac{1}{120}J_{4}(n)+\frac{1}{24}J_{2}(n)\right)\left(\frac{1}{840}J_{6}(n)+\frac{1}{60}J_{2}(n)\right)I^{2}(n),\\
C_{0,0,2}(n) & = & -\left(\frac{1}{1345344}J_{12}(n)+\frac{1}{7840}J_{6}(n)+\frac{1}{660}J_{2}(n)\right)I(n)\\
 &  & +\frac{1}{2}\left(\frac{1}{840}J_{6}(n)+\frac{1}{60}J_{2}(n)\right)^{2}I^{2}(n),\\
C_{0,3,0}(n) & = & \left(\frac{1}{2574000}J_{12}(n)+\frac{1}{42000}J_{8}(n)+\frac{17}{36000}J_{4}(n)+\frac{5}{2464}J_{2}(n)\right)I(n)\\
 &  & -\left(\frac{1}{16800}J_{8}(n)+\frac{1}{600}J_{4}(n)+\frac{5}{672}J_{2}(n)\right)\left(\frac{1}{120}J_{4}(n)+\frac{1}{24}J_{2}(n)\right)I^{2}(n)\\
 &  & +\frac{1}{6}\left(\frac{1}{120}J_{4}(n)+\frac{1}{24}J_{2}(n)\right)^{3}I^{3}(n).
\end{eqnarray*}
\end{thm}
\begin{proof}
When $n$ is odd, write
\[
\psi_{n}(x)=\sum_{\{(s,t):2s+3t\leq d(n)/2\}}\tilde{c}_{s,t}(n)b_{4}^{s}b_{6}^{t}x^{d(n)/2-(2s+3t)}.
\]
For this case, McKee \cite{8} proved a recurrence relation for $\tilde{c}_{s,t}(n)$:
\begin{eqnarray*}
 &  & (2s+3t)(2s+3t+\frac{1}{2})\tilde{c}_{s,t}(n)=\frac{1}{2}(\frac{n^{2}+3}{2}-2s-3t)(\frac{n^{2}}{6}-1+2s+3t)\tilde{c}_{s-1,t}(n)\\
 &  & -\frac{1}{4}(\frac{n^{2}+5}{2}-2s-3t)(\frac{n^{2}+3}{2}-2s-3t)\tilde{c}_{s,t-1}(n)+\frac{3}{2}(s+1)n^{2}\tilde{c}_{s+1,t-1}(n)-\frac{2}{3}(t+1)n^{2}\tilde{c}_{s-2,t+1}(n).
\end{eqnarray*}
By this formula, we can first calculate $\tilde{c}_{s,t}(n)$, and
then $C_{0,s,t}(n)$. The case when $n$ is even can be similarly
dealt with by another relation which is proved in the same paper.
\end{proof}
\begin{namedthm}{Remark}\textup{Clearly, these formulas have some
beautiful patterns, which we expect are shared by all $C_{r,s,t}(n)$.
Specifically, depending on $(r,s,t)$, each $C_{r,s,t}(n)$ is an
alternating sum of several components, each component is a product
of several subcomponents, and each subcomponent is a positive rational
linear combination of $J_{k}(n)$.}\end{namedthm}

If $j(E_{1}),j(E_{2})\neq0,1728$, it suffices to show that the map
\[
n\mapsto\left(D(n),\frac{C_{0,2,0}(n)}{C_{0,1,0}^{2}(n)},\frac{C_{0,1,1}(n)}{C_{0,1,0}(n)C_{0,0,1}(n)},\frac{C_{0,0,2}(n)}{C_{0,0,1}^{2}(n)}\right)
\]
is injective. Unfortunately, although it is supported by extensive
calculations, we fail to prove it. On the other hand, we have
\begin{thm}
Let $E_{1}$ and $E_{2}$ be two elliptic curves in generalized Weierstrass
form such that $\pi^{W}(E_{1}^{*}[n_{1}])=\pi^{W}(E_{2}^{*}[n_{2}])$
for some $n_{1}\neq n_{2}$, then $j(E_{1}),j(E_{2})\in\bar{\mathbb{Q}}\backslash\{0,1728\}$.\end{thm}
\begin{proof}
We have assumed $b_{2}=0$, so $j=0$ if and only if $b_{4}=0$, and
$j=1728$ if and only if $b_{6}=0$. Since $\pi^{W}(E_{1}^{*}[n_{1}])=\pi^{W}(E_{2}^{*}[n_{2}])$,
we have
\[
\begin{cases}
C_{0,1,0}(n_{1})b_{4}(E_{1})=C_{0,1,0}(n_{2})b_{4}(E_{2}),\\
C_{0,0,1}(n_{1})b_{6}(E_{1})=C_{0,0,1}(n_{2})b_{6}(E_{2}).
\end{cases}
\]
Thus there are three cases:\\
(A) $b_{4}(E_{1})=b_{4}(E_{2})=0$\\
We can assume $b_{6}(E_{1}),b_{6}(E_{2})\in\mathbb{Q}^{\times}$,
then $E_{1}$ and $E_{2}$ are both elliptic curves with complex multiplication
by $\mathcal{O}_{L}$, the integer ring of $L=\mathbb{Q}(\sqrt{-3})$.
By class field theory \cite{9} of imaginary quadratic fields \cite{13},
\[
L((\pi^{W}(E_{i}^{*}[n_{i}]))^{3})=L((\pi^{W}(E_{i}[n_{i}]))^{3})=L^{n_{i}}
\]
is the ray class field of $L$ for the modulus $n_{i}$. We want to
determine all $\mathfrak{m}\neq\mathfrak{n}$ such that $L^{\mathfrak{m}}=L^{\mathfrak{n}}$.
Since the common divisors of $\mathfrak{m}$ and $\mathfrak{n}$ give
the same ray class field, it is enough to assume $\mathfrak{m}|\mathfrak{n}$,
thus
\[
1=C_{L}^{\mathfrak{m}}/C_{L}^{\mathfrak{n}}=(I_{L}^{\mathfrak{m}}\cdot L^{\times}/L^{\times})/(I_{L}^{\mathfrak{n}}\cdot L^{\times}/L^{\times})=(I_{L}^{\mathfrak{m}}/(I_{L}^{\mathfrak{m}}\cap L^{\times}))/(I_{L}^{\mathfrak{n}}/(I_{L}^{\mathfrak{n}}\cap L^{\times})),
\]
so the indices
\[
[(I_{L}^{\mathfrak{m}}\cap L^{\times}):(I_{L}^{\mathfrak{n}}\cap L^{\times})]=[I_{L}^{\mathfrak{m}}:I_{L}^{\mathfrak{n}}]=\prod_{\mathfrak{p}}\left[U_{\mathfrak{p}}^{(v_{\mathfrak{p}}(\mathfrak{m}))}:U_{\mathfrak{p}}^{(v_{\mathfrak{p}}(\mathfrak{n}))}\right].
\]
In $L$, $2$ is inert, $3=\mathfrak{p}_{3}^{2}$ is ramified, and
$7=\mathfrak{p}_{7a}\mathfrak{p}_{7b}$ splits. Then
\[
I_{L}^{\mathfrak{m}}\cap L^{\times}=\begin{cases}
\mu_{6}, & \text{if }\mathfrak{m}=1,\\
\mu_{3}, & \text{if }\mathfrak{m}=\mathfrak{p}_{3},\\
\mu_{2}, & \text{if }\mathfrak{m}=2,\\
\mu_{1}, & \text{otherwise,}
\end{cases}
\]
where $\mu_{N}$ is the group of $N$-th roots of unity, and
\[
\prod_{\mathfrak{p}}\left[U_{\mathfrak{p}}^{(v_{\mathfrak{p}}(\mathfrak{m}))}:U_{\mathfrak{p}}^{(v_{\mathfrak{p}}(\mathfrak{n}))}\right]=\begin{cases}
2, & \text{if }(\mathfrak{m},\mathfrak{n})=(1,\mathfrak{p}_{3}),(2,2\mathfrak{p}_{3}),\\
3, & \text{if }(\mathfrak{m},\mathfrak{n})=(1,2),(\mathfrak{p}_{3},2\mathfrak{p}_{3}),\\
6, & \text{if }(\mathfrak{m},\mathfrak{n})=(1,2\mathfrak{p}_{3}),(1,\mathfrak{p}_{7a}),(1,\mathfrak{p}_{7b}).
\end{cases}
\]
We conclude that different moduli give different ray class fields
except for 
\[
L^{1}=L^{2}=L^{\mathfrak{p}_{3}}=L^{2\mathfrak{p}_{3}}=L^{\mathfrak{p}_{7a}}=L^{\mathfrak{p}_{7b}}.
\]
Thus the conductor of $L^{n}$ is $n/I(n)$, $L^{n_{1}}=L^{n_{2}}$
implies $n_{1}=n_{2}$.\\
(B) $b_{6}(E_{1})=b_{6}(E_{2})=0$\\
We can assume $b_{4}(E_{1}),b_{4}(E_{2})\in\mathbb{Q}^{\times}$,
then $E_{1}$ and $E_{2}$ are both elliptic curves with complex multiplication
by $\mathcal{O}_{L}$, the integer ring of $L=\mathbb{Q}(\sqrt{-1})$.
By class field theory of imaginary quadratic fields,
\[
L((\pi^{W}(E_{i}^{*}[n_{i}]))^{2})=L((\pi^{W}(E_{i}[n_{i}]))^{2})=L^{n_{i}}
\]
is the ray class field of $L$ for the modulus $n_{i}$. In $L$,
$2=\mathfrak{p}_{2}^{2}$ is ramified, and $5=\mathfrak{p}_{5a}\mathfrak{p}_{5b}$
splits. Then
\[
I_{L}^{\mathfrak{m}}\cap L^{\times}=\begin{cases}
\mu_{4}, & \text{if }\mathfrak{m}=1,\mathfrak{p}_{2},\\
\mu_{2}, & \text{if }\mathfrak{m}=2,\\
\mu_{1}, & \text{otherwise,}
\end{cases}
\]
and
\[
\prod_{\mathfrak{p}}\left[U_{\mathfrak{p}}^{(v_{\mathfrak{p}}(\mathfrak{m}))}:U_{\mathfrak{p}}^{(v_{\mathfrak{p}}(\mathfrak{n}))}\right]=\begin{cases}
1, & \text{if }\mathfrak{n}=\mathfrak{p}_{2}\mathfrak{m},\text{ where }\mathfrak{p}_{2}\nmid\mathfrak{m},\\
2, & \text{if }(\mathfrak{m},\mathfrak{n})=(1,2),(\mathfrak{p}_{2},2),(2,2\mathfrak{p}_{2}),\\
4, & \text{if }(\mathfrak{m},\mathfrak{n})=(1,2\mathfrak{p}_{2}),(\mathfrak{p}_{2},2\mathfrak{p}_{2}),(1,\mathfrak{p}_{5a}),(1,\mathfrak{p}_{5b}),\\
 & (1,\mathfrak{p}_{2}\mathfrak{p}_{5a}),(1,\mathfrak{p}_{2}\mathfrak{p}_{5b}),(\mathfrak{p}_{2},\mathfrak{p}_{2}\mathfrak{p}_{5a}),(\mathfrak{p}_{2},\mathfrak{p}_{2}\mathfrak{p}_{5b}).
\end{cases}
\]
We conclude that different moduli give different ray class fields
except for
\[
\begin{cases}
L^{1}=L^{\mathfrak{p}_{2}}=L^{2}=L^{2\mathfrak{p}_{2}}=L^{\mathfrak{p}_{5a}}=L^{\mathfrak{p}_{5b}}=L^{\mathfrak{p}_{2}\mathfrak{p}_{5a}}=L^{\mathfrak{p}_{2}\mathfrak{p}_{5b}},\\
L^{\mathfrak{m}}=L^{\mathfrak{p}_{2}\mathfrak{m}},\text{ where }\mathfrak{p}_{2}\nmid\mathfrak{m}.
\end{cases}
\]
Thus the conductor of $L^{n}$ is $n/I(n)$, $L^{n_{1}}=L^{n_{2}}$
implies $n_{1}=n_{2}$.\\
(C) $b_{4}(E_{1}),b_{4}(E_{2}),b_{6}(E_{1}),b_{6}(E_{2})\neq0$\\
The relation
\[
\sum_{0\leq k\leq s/3}C_{0,s-3k,2k}(n_{1})b_{4}^{s-3k}(E_{1})b_{6}^{2k}(E_{1})=\sum_{0\leq k\leq s/3}C_{0,s-3k,2k}(n_{2})b_{4}^{s-3k}(E_{2})b_{6}^{2k}(E_{2})
\]
implies
\[
\sum_{0\leq k\leq s/3}\left(C_{0,s-3k,2k}(n_{1})-\frac{C_{0,s-3k,2k}(n_{2})C_{0,1,0}^{s-3k}(n_{1})C_{0,0,1}^{2k}(n_{1})}{C_{0,1,0}^{s-3k}(n_{2})C_{0,0,1}^{2k}(n_{2})}\right)\left(\frac{b_{6}^{2}(E_{1})}{b_{4}^{3}(E_{1})}\right)^{k}=0.
\]
Now $b_{6}^{2}(E_{1})/b_{4}^{3}(E_{1})\in\bar{\mathbb{Q}}$ unless
all of the coefficients are zero. Part (B) implies that
\[
\sum_{\{s:2s\leq D(n_{1})\}}C_{0,s,0}(n_{1})b_{4}^{s}(E_{1})x^{D(n_{1})-2s}\neq\sum_{\{s:2s\leq D(n_{2})\}}C_{0,s,0}(n_{2})b_{4}^{s}(E_{2})x^{D(n_{2})-2s}
\]
as polynomials, so there exists some $s$ such that
\[
C_{0,s,0}(n_{1})b_{4}^{s}(E_{1})\neq C_{0,s,0}(n_{2})b_{4}^{s}(E_{2}),
\]
which in turn implies the constant term
\[
C_{0,s,0}(n_{1})-\frac{C_{0,s,0}(n_{2})C_{0,1,0}^{s}(n_{1})}{C_{0,1,0}^{s}(n_{2})}\neq0.
\]
Thus $b_{6}^{2}(E_{1})/b_{4}^{3}(E_{1})\in\bar{\mathbb{Q}}$, which
implies $j(E_{1}),j(E_{2})\in\bar{\mathbb{Q}}$.
\end{proof}

\section{\label{section 3}Intersection of Projective Torsion Points}

Now let $K$ be a number field, $G_{K}=\text{Gal}(\bar{K}/K)$ the
absolute Galois group of $K$,
\[
\chi_{K}:G_{K}\twoheadrightarrow\text{Gal}(K^{cyc}/K)\cong\text{Gal}(\mathbb{Q}^{cyc}/K\cap\mathbb{Q}^{cyc})\hookrightarrow\text{Gal}(\mathbb{Q}^{cyc}/\mathbb{Q})\cong\hat{\mathbb{Z}}^{\times}
\]
the cyclotomic character of $K$. Consider the associated Galois representation
\[
\rho_{E}:G_{K}\rightarrow\text{Aut}(E[\infty])\cong\text{GL}(2,\hat{\mathbb{Z}}).
\]
Since $\chi_{K}=\text{det}\circ\rho_{E}$, we always have $\rho_{E}(G_{K^{cyc}})\subseteq\text{SL}(2,\hat{\mathbb{Z}}).$
Zywina \cite{16} proved that if $K\neq\mathbb{Q}$, then for almost
all elliptic curves defined over $K$, this is actually an equality,
namely, $\rho_{E}(G_{K^{cyc}})=\text{SL}(2,\hat{\mathbb{Z}}).$
\begin{thm}
\label{theorem 12}Given $(E_{1},O_{1},\pi_{1})$ and $(E_{2},O_{2},\pi_{2})$,
all defined over a number field $K\neq\mathbb{Q}$, if\\
\textup{(A)} $\pi_{1}(O_{1})=\pi_{2}(O_{2})$,\\
\textup{(B)} corollary \ref{corollary 6} does not hold,\\
\textup{(C)} $\rho_{E_{1}}(G_{K^{cyc}})=\rho_{E_{2}}(G_{K^{cyc}})=\textup{SL}(2,\hat{\mathbb{Z}})$,\\
then
\[
\#\pi_{1}(E_{1}[\infty])\cap\pi_{2}(E_{2}[\infty])=1.
\]
\end{thm}
\begin{proof}
Suppose that there exists $a\in\pi_{1}(E_{1}^{*}[n_{1}])\cap\pi_{2}(E_{2}^{*}[n_{2}])$
for some $n_{1},n_{2}>1$. By assumption (C), $G_{K^{cyc}}$ transitively
acts on $\pi_{i}(E_{i}^{*}[n_{i}])$, which is therefore the orbit
of $a$, and the degree of $a$ is given by $D(n_{i})$. We also have
the exact sequence
\[
1\rightarrow G_{K^{cyc}(\pi_{i}(E_{i}^{*}[n_{i}]))}\rightarrow G_{K^{cyc}}\rightarrow\text{PSL}(2,\mathbb{Z}/n_{i}\mathbb{Z})\rightarrow1,
\]
as a consequence,
\[
[K^{cyc}(\text{orbit of }a):K^{cyc}(a)]=\frac{[G_{K^{cyc}}:G_{K^{cyc}(\text{orbit of }a)}]}{[K^{cyc}(a):K^{cyc}]}=\frac{\#\text{PSL}(2,\mathbb{Z}/n_{i}\mathbb{Z})}{D(n_{i})}=n_{i},
\]
which implies that $n_{1}=n_{2}$, hence contradicts (A) and (B).
\end{proof}
For any $\sigma\in G_{K}$, $\sigma(\sqrt{\Delta})=\epsilon(\rho_{E}(\sigma))\sqrt{\Delta}$,
where
\[
\epsilon:\text{GL}(2,\hat{\mathbb{Z}})\rightarrow\text{GL}(2,\mathbb{Z}/2\mathbb{Z})\rightarrow\{\pm1\}
\]
is the signature character. If $K=\mathbb{Q}$, then $\sqrt{\Delta}\in\mathbb{Q}^{ab}=\mathbb{Q}^{cyc}$,
so $\rho_{E}(G_{\mathbb{Q}^{cyc}})\subseteq\text{SL}(2,\hat{\mathbb{Z}})\cap\text{ker}(\epsilon)$,
a subgroup of index $2$ in $\text{SL}(2,\hat{\mathbb{Z}})$. Jones
\cite{6} proved that for almost all elliptic curves defined over
$\mathbb{Q}$, this is actually an equality, namely, $\rho_{E}(G_{\mathbb{Q}^{cyc}})=\text{SL}(2,\hat{\mathbb{Z}})\cap\text{ker}(\epsilon)$.
\begin{thm}
\label{theorem 13}Given $(E_{1},O_{1},\pi_{1})$ and $(E_{2},O_{2},\pi_{2})$,
all defined over $\mathbb{Q}$, if\\
\textup{(A)} $\pi_{1}(O_{1})=\pi_{2}(O_{2})$,\\
\textup{(B)} corollary \ref{corollary 6} does not hold,\\
\textup{(C)} $\rho_{E_{1}}(G_{K^{cyc}})=\rho_{E_{2}}(G_{K^{cyc}})=\textup{SL}(2,\hat{\mathbb{Z}})\cap\textup{ker}(\epsilon)$,\\
then
\[
\#\pi_{1}(E_{1}[\infty])\cap\pi_{2}(E_{2}[\infty])=1.
\]
\end{thm}
\begin{proof}
The proof is nearly the same as the case $K\neq\mathbb{Q}$, except
that
\[
[\mathbb{Q}^{cyc}(\text{orbit of }a):\mathbb{Q}^{cyc}(a)]=\begin{cases}
n_{i}, & \text{if }n_{i}\text{ is odd},\\
n_{i}/2, & \text{if }n_{i}\text{ is even}.
\end{cases}
\]
If $n_{1}\neq n_{2}$, then $n_{1}=n_{2}/2$ is odd, but which implies
$D(n_{1})=D(n_{2})/3$, a contradiction.
\end{proof}
\begin{namedthm}{Remark}\textup{If question \ref{question 8} has
an affirmative answer, then the assumption (C) in theorem \ref{theorem 12}
and \ref{theorem 13} can be weakened by assuming that $G_{K}$ acts
on $\pi_{i}(E_{i}^{*}[n_{i}])$ transitively. This suggests that our
conjecture \ref{conjeture 2}, apparently an unlikely intersection
type problem \cite{15}, might be somewhat related to Serre's uniformity
conjecture \cite{11}.}\end{namedthm}

With the condition $\pi_{1}(O_{1})=\pi_{2}(O_{2})$ dropped, more
intersection points can be obtained. We begin with the classical results
for $3$-torsion points and $4$-torsion points.
\begin{prop}
\label{proposition 15}For any elliptic curve $E$ defined over $K$,
there exists an even morphism $\pi\in\bar{K}(E)$ of degree $2$ such
that $\pi(E^{*}[3])=\{\infty,1,\rho,\rho^{2}\}$, where $\rho$ is
a primitive cube root of unity.\end{prop}
\begin{proof}
Consider the family $E_{\lambda}:x^{3}+y^{3}+z^{3}=3\lambda xyz$,
which is an elliptic curve provided $\lambda^{3}\neq1$. Its $3$-torsion
points are
\[
\begin{Bmatrix}(1:-1:0), & (1:-\rho:0), & (1:-\rho^{2}:0),\\
(0:1:-1), & (0:1:-\rho), & (0:1:-\rho^{2}),\\
(-1:0:1), & (-\rho:0:1), & (-\rho^{2}:0:1).
\end{Bmatrix}
\]
If we take the origin to be $O_{\lambda}=(1:-1:0)$, then the projection
map $\pi_{\lambda}:E_{\lambda}\rightarrow\mathbb{P}^{1},(x:y:z)\mapsto-(x+y)/z$
maps $E_{\lambda}^{*}[3]$ to the desired set. Since every elliptic
curve defined over $K$ can be transformed to some $E_{\lambda}$
via an isomorphism defined over $\bar{K}$, the conclusion follows.\end{proof}
\begin{prop}
\label{proposition 16}For any elliptic curve $E$ defined over $K$,
there exists an even morphism $\pi\in\bar{K}(E)$ of degree $2$ such
that $\pi(E^{*}[4])=\{0,\infty,\pm1,\pm i\}$, where $i$ is a primitive
fourth root of unity.\end{prop}
\begin{proof}
Consider the family $E_{\delta}:y^{2}=x^{4}-(\delta^{2}+1/\delta^{2})x^{2}+1$,
which is a curve of genus $1$ with a unique singularity at $(0:1:0)$
provided $\delta^{4}\neq0,1$. Its $2$-torsion points $(\pm\delta^{\pm1},0)$
respectively induce $x\mapsto\pm x^{\pm1}\in\text{Aut}(\mathbb{\mathbb{P}}^{1})$
if we take $O_{\delta}=(\delta,0)$, and $\pi_{\delta}:E_{\delta}\rightarrow\mathbb{P}^{1},(x,y)\mapsto x$.
The fixed points of those nontrivial ones, $\{0,\infty\}$, $\{\pm1\},$
and $\{\pm i\}$, therefore constitute the collection of $\pi_{\delta}(E_{\delta}^{*}[4])$.
Since every elliptic curve defined over $K$ can be transformed to
some $E_{\delta}$ via a birational isomorphism defined over $\bar{K}$,
the conclusion follows.
\end{proof}
Proposition \ref{proposition 15} and \ref{proposition 16} indicate
that for any elliptic curve, the projective $3$-torsion points are
equivalent to the vertices of a regular tetrahedron, while the projective
$4$-torsion points are equivalent to the vertices of a regular octahedron.
It is therefore tempting to guess the projective $5$-torsion points
are equivalent to the vertices of a regular icosahedron. However,
this is not the case. Following Klein \cite{7}, consider the family
\begin{eqnarray*}
 &  & E_{s,t}:y^{2}=x^{3}-3(s^{20}+228s^{15}t^{5}+494s^{10}t^{10}-228s^{5}t^{15}+t^{20})x\\
 &  & +2(s^{30}-522s^{25}t^{5}-10005s^{20}t^{10}-10005s^{10}t^{20}+522s^{5}t^{25}+t^{30}).
\end{eqnarray*}
Its projective $5$-torsion points can be explicitly expressed as
\begin{eqnarray*}
x_{\infty}^{+} & = & -\left[\left(5+\frac{6}{\sqrt{5}}\right)s^{10}-\frac{66}{\sqrt{5}}s^{5}t^{5}+\left(5-\frac{6}{\sqrt{5}}\right)t^{10}\right],\\
x_{\infty}^{-} & = & -\left[\left(5-\frac{6}{\sqrt{5}}\right)s^{10}+\frac{66}{\sqrt{5}}s^{5}t^{5}+\left(5+\frac{6}{\sqrt{5}}\right)t^{10}\right],\\
x_{k}^{+} & = & (s^{10}+30s^{5}t^{5}+t^{10})+(12s^{9}t+24s^{4}t^{6})\omega^{k}+(24s^{8}t^{2}-12s^{3}t^{7})\omega^{2k}\\
 &  & +(36s^{7}t^{3}+12s^{2}t^{8})\omega^{3k}+60s^{6}t^{4}\omega^{4k},\\
x_{k}^{-} & = & (s^{10}-30s^{5}t^{5}+t^{10})+(24s^{6}t^{4}-12st^{9})\omega^{k}+(12s^{7}t^{3}+24s^{2}t^{8})\omega^{2k}\\
 &  & +(12s^{8}t^{2}-36s^{3}t^{7})\omega^{3k}+60s^{4}t^{6}\omega^{4k},
\end{eqnarray*}
where $\omega$ is a primitive fifth root of unity, and $k\in\mathbb{Z}/5\mathbb{Z}$.
Since the cross ratio
\[
\frac{(x_{\infty}^{+}-x_{0}^{-})(x_{0}^{+}-x_{\infty}^{-})}{(x_{\infty}^{+}-x_{\infty}^{-})(x_{0}^{+}-x_{0}^{-})}=\frac{(s^{2}-st+\frac{3-\sqrt{5}}{2}t^{2})(s^{2}+\frac{3+\sqrt{5}}{2}st+\frac{3+\sqrt{5}}{2}t^{2})}{\sqrt{5}st(s^{2}-st-t^{2})}
\]
is not a constant, the projective $5$-torsion points are not equivalent
to any fixed collection of $12$ points.

\medskip{}

Although any single $\pi(E^{*}[2])$, $\pi(E^{*}[3])$, or $\pi(E^{*}[4])$
is insufficient to determine $\pi(E[\infty])$, any pair of them can
do so.
\begin{cor}
Given $(E_{1},O_{1},\pi_{1})$ and $(E_{2},O_{2},\pi_{2})$, if\\
\textup{(A)} $\pi_{1}(E_{1}^{*}[2])=\pi_{2}(E_{2}^{*}[2])$ and $\pi_{1}(E_{1}^{*}[3])=\pi_{2}(E_{2}^{*}[3])$,
or\\
\textup{(B)} $\pi_{1}(E_{1}^{*}[2])=\pi_{2}(E_{2}^{*}[2])$ and $\pi_{1}(E_{1}^{*}[4])=\pi_{2}(E_{2}^{*}[4])$,
or\\
\textup{(C)} $\pi_{1}(E_{1}^{*}[3])=\pi_{2}(E_{2}^{*}[3])$ and $\pi_{1}(E_{1}^{*}[4])=\pi_{2}(E_{2}^{*}[4])$,\\
then $\pi_{1}(O_{1})=\pi_{2}(O_{2})$. In particular, corollary \ref{corollary 6}
holds.\end{cor}
\begin{proof}
(A) It suffices to show that $\pi(E^{*}[2])$ and $\pi(E^{*}[3])$
determine $\pi(O)$. By proposition \ref{proposition 15}, we can
assume $\pi(E^{*}[3])=\{\infty,1,\rho,\rho^{2}\}$. Let $\lambda=\pi(O)$,
and consider $(E_{\lambda},O_{\lambda},\pi_{\lambda})$. Since
\[
\pi_{\lambda}:E_{\lambda}\rightarrow\mathbb{P}^{1},(x:y:z)\mapsto-\frac{x+y}{z}=\frac{z^{2}-3\lambda xy}{x^{2}-xy+y^{2}},
\]
we have $\pi_{\lambda}(O_{\lambda})=\lambda$. By corollary \ref{corollary 6},
$\pi_{\lambda}(E_{\lambda}^{*}[2])=\pi(E^{*}[2])$. Thus we just need
to show that $\pi_{\lambda}(E_{\lambda}^{*}[2])$ determines $\lambda$.
Since the points in $\pi_{\lambda}(E_{\lambda}^{*}[2])$ are the roots
of $x^{3}+3\lambda x^{2}-4$, this is clear.\\
(B) $\pi_{\delta}(E_{\delta}^{*}[2])=\{-\delta,1/\delta,-1/\delta\}$
determines $\delta$.\\
(C) We need to show that $\pi_{\delta}(E_{\delta}^{*}[3])$ determines
$\delta$. The nonsingular model of $E_{\delta}$ is
\[
E_{\delta}^{ns}:Y^{2}=X(X-1)(X-\frac{1}{4}(\delta+\frac{1}{\delta})^{2}),
\]
where
\[
X=\frac{(\delta^{2}+1)(\delta x-1)}{2\delta(x-\delta)},Y=\frac{(\delta^{4}-1)y}{4\delta(x-\delta)^{2}}.
\]
The division polynomials of $E_{\delta}^{ns}$ imply the division
polynomials of $E_{\delta}$ via this birational isomorphism. In particular,
the points in $\pi_{\delta}(E_{\delta}^{*}[3])$ are the roots of
$x^{4}+2\delta x^{3}-(2/\delta)x-1$, then the result is immediate.
\end{proof}
Now we are ready to give the main result of this article. For any
$E_{\delta_{1}}$ and $E_{\delta_{2}}$, the intersection of $\pi_{\delta_{1}}(E_{\delta_{1}}[\infty])$
and $\pi_{\delta_{2}}(E_{\delta_{2}}[\infty])$ has at least $6$
elements. If it contains another element $a$, then as we have seen
in the proof of proposition \ref{proposition 16}, it also contains
$-a$, $1/a$, and $-1/a$. Therefore, its cardinality must be $4k+6$.
If we fix $\delta_{1}$ and vary $\delta_{2}$, then $k=1$ can be
easily attained. In the next theorem, we construct an example to improve
$k=2$.
\begin{thm}
\label{theorem 18}
\[
\underset{\{(K,E_{1},O_{1},\pi_{1},E_{2},O_{2},\pi_{2}):\pi_{1}(E_{1}[2])\neq\pi_{2}(E_{2}[2])\}}{\textup{sup}}\#\pi_{1}(E_{1}[\infty])\cap\pi_{2}(E_{2}[\infty])\geq14
\]
\end{thm}
\begin{proof}
From the division polynomials of $E_{\delta}^{ns}$, we know that
the third and fifth primitive division polynomials of $E_{\delta}$
are
\[
\begin{cases}
\tilde{F}_{3}(x,\delta)=2x^{3}\delta^{2}+(x^{4}-1)\delta-2x,\\
\\
\tilde{F}_{5}(x,\delta)=8x^{5}\delta^{6}-4x^{6}(x^{4}-1)\delta^{5}-2x^{3}(x^{8}+6x^{4}+5)\delta^{4}\\
+(x^{12}+5x^{8}-5x^{4}-1)\delta^{3}+2x(5x^{8}+6x^{4}+1)\delta^{2}-4x^{2}(x^{4}-1)\delta-8x^{7}.
\end{cases}
\]
Now we want to find $\delta_{1}$ and $\delta_{2}$ such that there
exist $u\in\pi_{\delta_{1}}(E_{\delta_{1}}^{*}[3])\cap\pi_{\delta_{2}}(E_{\delta_{2}}^{*}[3])$,
and $v\in\pi_{\delta_{1}}(E_{\delta_{1}}^{*}[5])\cap\pi_{\delta_{2}}(E_{\delta_{2}}^{*}[5])$.
In other words, $\tilde{F}_{3}(u,\delta_{1})=\tilde{F}_{3}(u,\delta_{2})=0$,
and $\tilde{F}_{5}(v,\delta_{1})=\tilde{F}_{5}(v,\delta_{2})=0$.
Since $\tilde{F}_{3}(x,\delta)$ is a quadratic polynomial in $\delta$,
any fixed $u$ such that $u^{4}\neq0,1$, and $u^{8}+14u^{4}+1\neq0$
gives exactly two roots satisfying $\delta_{1}^{4},\delta_{2}^{4}\neq0,1$,
and $\delta_{1}\neq\pm\delta_{2}^{\pm1}$. Since $\delta_{1}$ and
$\delta_{2}$ are also two roots of $\tilde{F}_{5}(v,\delta)=0$,
we have $\tilde{F}_{3}(u,\delta)|\tilde{F}_{5}(v,\delta)$ as polynomials
in $\delta$. By long division, this is equivalent to require
\[
\begin{cases}
C_{0}:u^{15}v^{10}-u^{14}v^{11}-u^{13}v^{12}+u^{16}v^{5}-u^{15}v^{6}-22u^{14}v^{7}-5u^{13}v^{8}+20u^{12}v^{9}\\
+5u^{11}v^{10}-2u^{10}v^{11}+u^{9}v^{12}-5u^{14}v^{3}+5u^{13}v^{4}+32u^{12}v^{5}-5u^{11}v^{6}-12u^{10}v^{7}\\
+5u^{9}v^{8}-5u^{7}v^{10}-u^{6}v^{11}+u^{13}+4u^{12}v-10u^{10}v^{3}-5u^{9}v^{4}-2u^{8}v^{5}\\
+5u^{7}v^{6}-6u^{6}v^{7}-u^{3}v^{10}-u^{9}-5u^{6}v^{3}+8u^{4}v^{5}+u^{3}v^{6}+v^{5}=0,\\
\\
C_{1}:u^{19}v^{10}-u^{18}v^{11}-u^{17}v^{12}+u^{20}v^{5}-u^{19}v^{6}-6u^{18}v^{7}-5u^{17}v^{8}+20u^{16}v^{9}+8u^{15}v^{10}\\
-5u^{14}v^{11}-2u^{13}v^{12}-5u^{18}v^{3}+5u^{17}v^{4}+35u^{16}v^{5}+8u^{15}v^{6}-30u^{14}v^{7}-10u^{13}v^{8}-20u^{12}v^{9}\\
-2u^{11}v^{10}+5u^{10}v^{11}-u^{9}v^{12}+u^{17}+4u^{16}v-16u^{15}v^{2}-25u^{14}v^{3}+10u^{13}v^{4}-14u^{12}v^{5}\\
+2u^{11}v^{6}+30u^{10}v^{7}-5u^{9}v^{8}+8u^{7}v^{10}+u^{6}v^{11}+2u^{13}-4u^{12}v+25u^{10}v^{3}+5u^{9}v^{4}\\
-10u^{8}v^{5}-8u^{7}v^{6}+6u^{6}v^{7}+u^{3}v^{10}+u^{9}+5u^{6}v^{3}-11u^{4}v^{5}-u^{3}v^{6}-v^{5}=0.
\end{cases}
\]
Considering $C_{0}$ and $C_{1}$ as polynomials in $v$, then their
resultant is
\[
-2^{48}u^{204}(u^{4}-1)^{36}(32u^{24}+1369u^{20}+18812u^{16}+90646u^{12}+18812u^{8}+1369u^{4}+32),
\]

whose roots are the $u$-coordinates of their common points. Let $u$
such that $u^{4}\neq0,1$ be any nontrivial root, $(u,v)$ the corresponding
common point, $\delta_{1}$ and $\delta_{2}$ the roots of $\tilde{F}_{3}(u,\delta)=0$,
then we have
\[
\begin{cases}
\pi_{\delta_{1}}(E_{\delta_{1}}[2])\neq\pi_{\delta_{2}}(E_{\delta_{2}}[2]),\pi_{\delta_{1}}(E_{\delta_{1}}^{*}[4])=\pi_{\delta_{2}}(E_{\delta_{2}}^{*}[4])=\{0,\infty,\pm1,\pm i\},\\
u\in\pi_{\delta_{1}}(E_{\delta_{1}}^{*}[3])\cap\pi_{\delta_{2}}(E_{\delta_{2}}^{*}[3]),-u,1/u,-1/u\in\pi_{\delta_{1}}(E_{\delta_{1}}^{*}[6])\cap\pi_{\delta_{2}}(E_{\delta_{2}}^{*}[6]),\\
v\in\pi_{\delta_{1}}(E_{\delta_{1}}^{*}[5])\cap\pi_{\delta_{2}}(E_{\delta_{2}}^{*}[5]),-v,1/v,-1/v\in\pi_{\delta_{1}}(E_{\delta_{1}}^{*}[10])\cap\pi_{\delta_{2}}(E_{\delta_{2}}^{*}[10]),
\end{cases}
\]
so the supremum is at least $14$.
\end{proof}
\begin{namedthm}{Remark}\textup{\label{remark 19}After submitting
this article, we have successfully improved the previous result from
$14$ points (using $3$-torsion points and $5$-torsion points) to
$22$ points (using $3$-torsion points and $7$-torsion points).
We will give full details in the next publication.}\end{namedthm}

\begin{namedthm}{Remark}\textup{The intersection can also be investigated
using Tate's explicit parametrization of elliptic curves over $p$-adic
fields \cite{13}. We plan to explore the details of this approach
in the future.}\end{namedthm}

\section{Appendix: Jordan's Totient Function}

We are concerned with the values taken by $J_{k}(n)$ as well. $J_{1}(n)$
is simply Euler's totient function, for which we have the famous Ford's
theorem \cite{4} and Carmichael's conjecture \cite{2,3}. $J_{2}(n)$
is far from being injective, which prevents a simple answer to question
\ref{question 8}. Even their combination is not injective, since
\begin{eqnarray*}
 &  & J_{1}(15)=J_{1}(16)=8,\\
 &  & J_{2}(15)=J_{2}(16)=192.
\end{eqnarray*}
It is quite surprising that $J_{3}(n)$ is still not injective, the
smallest identical pair is
\[
J_{3}(28268)=J_{3}(28710)=19764446869440.
\]
We do not know any identical pair for $k\geq4$, and expect that such
coincidence should be very rare.

\medskip{}

Finally, let us conclude this article with a collection of partial
results:
\begin{prop}
Let $p,q,p_{1}\neq p_{2},q_{1}\neq q_{2}$ be primes, $v_{p}(n)$
the power of $p$ in the prime decomposition of $n$, $\omega(n)$
the number of distinct prime divisors of $n$, then\\
\textup{(A)} If $J_{2}(p^{s})=J_{2}(q^{t})$, then $p^{s}=q^{t}$
or $\{p^{s},q^{t}\}=\{7,8\}$;\\
\textup{(B)} If $J_{k}(p_{1}p_{2})=J_{k}(q_{1}q_{2}),k=2,4$, then
$p_{1}p_{2}=q_{1}q_{2}$;\\
\textup{(C)} If $J_{k}(m)=J_{k}(n),k=2,4,6$, then $v_{p}(m)=v_{p}(n)$
for $p=2,3$, $v_{p}(m)=v_{p}(n)$ or $\{v_{p}(m),v_{p}(n)\}=\{0,1\}$
for any other $p\not\equiv1\textup{ mod }12$, and $\omega(m)=\omega(n)$.\\
\textup{(D)} If $J_{k}(m)=J_{k}(n)$ for infinitely many $k$, then
$m=n$.\end{prop}
\begin{proof}
(A) Now we have $p^{2s-2}(p^{2}-1)=q^{2t-2}(q^{2}-1)$. If $p,q\neq2$
and $s,t\neq1$, then $p$ (resp. $q$) is the largest prime divisor
of left-hand side (resp. right-hand side). Hence $p=q$ and $s=t$.
If $p=2$, then $3\cdot2^{2s-2}=q^{2t-2}(q^{2}-1)$. Since $4$ cannot
be a common divisor of $q+1$ and $q-1$, one of them must be a divisor
of $6$. Hence $q=2$ and $s=t$, or $\{p^{s},q^{t}\}=\{7,8\}$. If
$s=1$, then $p^{2}-1=q^{2t-2}(q^{2}-1)$. We can assume $q\neq2$,
then $q$ cannot be a common divisor of $p+1$ and $p-1$, so $q^{2t-2}\leq p\pm1=(p\mp1)\pm2\leq(q^{2}-1)\pm2$,
so $t=1$ and $p=q$, or $t=2$. If $t=2$, then either $q^{2}=p+1=(p-1)+2=(q^{2}-1)+2$
or $2q^{2}\leq p+1=(p-1)+2\leq(q^{2}-1)/2+2$, neither is possible.\\
(B) It is straightforward that $(p_{1}^{2}-1)(p_{2}^{2}-1)=(q_{1}^{2}-1)(q_{2}^{2}-1)$
and $(p_{1}^{4}-1)(p_{2}^{4}-1)=(q_{1}^{4}-1)(q_{2}^{4}-1)$ imply
$p_{1}p_{2}=q_{1}q_{2}$.\\
(C) Since $2\nmid p^{4}+p^{2}+1=(p^{2}+p+1)(p^{2}-p+1)$, from $J_{6}(n)/J_{2}(n)=\prod_{p|n}p^{4v_{p}(n)-4}(p^{4}+p^{2}+1)$,
we can see $v_{2}(m)=v_{2}(n)$ or $\{v_{2}(m),v_{2}(n)\}=\{0,1\}$.
Since $3\nmid p^{2}+1$, from $J_{4}(n)/J_{2}(n)=\prod_{p|n}p^{2v_{p}(n)-2}(p^{2}+1)$,
we can see $v_{3}(m)=v_{3}(n)$ or $\{v_{3}(m),v_{3}(n)\}=\{0,1\}$.
This strategy works for any other $p\not\equiv1\text{ mod }12$, since
either $-1$ or $-3$ is a quadratic nonresidue. If $v_{2}(m)=0$
and $v_{2}(n)=1$, then $5/3=(2^{2}+1)/(2^{2}-1)\leq J_{4}(n)/J_{2}^{2}(n)=J_{4}(m)/J_{2}^{2}(m)\leq\prod_{p\geq3}(p^{2}+1)/(p^{2}-1)=3/2$,
a contradiction. If $v_{3}(m)=0$ and $v_{3}(n)=1$, then $5/4=(3^{2}+1)/(3^{2}-1)\leq J_{4}(n)/J_{2}^{2}(n)=J_{4}(m)/J_{2}^{2}(m)\leq\prod_{p\geq5}(p^{2}+1)/(p^{2}-1)=6/5$,
again a contradiction. Moreover, since $2|p^{2}+1$, but $4\nmid p^{2}+1$,
from $J_{4}(n)/J_{2}(n)$, we can see $\omega(m)=\omega(n)$.\\
(D) Assume $k>1$, then $m^{k}/\zeta(k)<J_{k}(m)=J_{k}(n)\leq n^{k}$,
thus $\zeta(k)^{-1/k}<m/n<\zeta(k)^{1/k}$ by symmetry. Since $\zeta(k)^{1/k}\to1$
as $k\to\infty$, we have $m=n$.
\end{proof}
\textbf{Acknowledgments.} The first author acknowledges that the article
was prepared within the framework of a subsidy granted to the HSE
by the Government of the Russian Federation for the implementation
of the Global Competitiveness Program. The first author was also supported
by a Simons Travel Grant. The second author was supported by the MacCracken
Program offered by New York University.

Fedor Bogomolov\\
Courant Institute of Mathematical Sciences, New York University\\
251 Mercer Street, New York, NY 10012, USA\\
Email: bogomolo@cims.nyu.edu

\medskip{}

Fedor Bogomolov\\
Laboratory of Algebraic Geometry and its Applications\\
National Research University Higher School of Economics\\
7 Vavilova Street, 117312 Moscow, Russia

\medskip{}

Hang Fu\\
Courant Institute of Mathematical Sciences, New York University\\
251 Mercer Street, New York, NY 10012, USA\\
Email: fu@cims.nyu.edu

\begin{thebibliography}{10}
\bibitem{1}\addcontentsline{toc}{section}{\bibname}Fedor Bogomolov
and Yuri Tschinkel, Algebraic varieties over small fields. Diophantine
geometry, 73\textendash 91, CRM Series, 4, Ed. Norm., Pisa, 2007.

\bibitem{2}Robert Daniel Carmichael, On Euler's $\phi$-function.
Bull. Amer. Math. Soc. 13 (1907), no. 5, 241\textendash 243.

\bibitem{3}Robert Daniel Carmichael, Note on Euler's $\varphi$-function.
Bull. Amer. Math. Soc. 28 (1922), no. 3, 109\textendash 110.

\bibitem{4}Kevin Ford, The number of solutions of $\phi(x)=m$. Ann.
of Math. (2) 150 (1999), no. 1, 283\textendash 311.

\bibitem{5}Robin Hartshorne, Algebraic geometry. Graduate Texts in
Mathematics, No. 52. Springer-Verlag, New York-Heidelberg, 1977. xvi+496
pp. ISBN: 0-387-90244-9

\bibitem{6}Nathan Jones, Almost all elliptic curves are Serre curves.
Trans. Amer. Math. Soc. 362 (2010), no. 3, 1547\textendash 1570.

\bibitem{7}Felix Klein, Lectures on the icosahedron and the solution
of equations of the fifth degree. Translated into English by George
Gavin Morrice. Second and revised edition. Dover Publications, Inc.,
New York, N.Y., 1956. xvi+289 pp.

\bibitem{8}James McKee, Computing division polynomials. Math. Comp.
63 (1994), no. 208, 767\textendash 771.

\bibitem{9}J\"urgen Neukirch, Class field theory. Grundlehren der
Mathematischen Wissenschaften, 280. Springer-Verlag, Berlin, 1986.
viii+140 pp. ISBN: 3-540-15251-2

\bibitem{10}Michel Raynaud, Courbes sur une vari\'et\'e ab\'elienne
et points de torsion. Invent. Math. 71 (1983), no. 1, 207\textendash 233.

\bibitem{11}Jean-Pierre Serre, Propri\'et\'es galoisiennes des
points d'ordre fini des courbes elliptiques. Invent. Math. 15 (1972),
no. 4, 259\textendash 331.

\bibitem{12}Joseph Silverman, The arithmetic of elliptic curves.
Second edition. Graduate Texts in Mathematics, 106. Springer, Dordrecht,
2009. xx+513 pp. ISBN: 978-0-387-09493-9

\bibitem{13}Joseph Silverman, Advanced topics in the arithmetic of
elliptic curves. Graduate Texts in Mathematics, 151. Springer-Verlag,
New York, 1994. xiv+525 pp. ISBN: 0-387-94328-5

\bibitem{14}Wolfram Research, Inc., Mathematica, Version 10.0, Champaign,
IL (2014).

\bibitem{15}Umberto Zannier, Some problems of unlikely intersections
in arithmetic and geometry. With appendixes by David Masser. Annals
of Mathematics Studies, 181. Princeton University Press, Princeton,
NJ, 2012. xiv+160 pp. ISBN: 978-0-691-15371-1

\bibitem{16}David Zywina, Elliptic curves with maximal Galois action
on their torsion points. Bull. Lond. Math. Soc. 42 (2010), no. 5,
811\textendash 826.

\end{thebibliography}
\end{document}